\documentclass[11pt]{amsart}
\usepackage{amssymb}
\usepackage{amsmath}
\usepackage{enumitem}
\usepackage{mathrsfs}
\usepackage[all]{xy}
\usepackage{marginnote}

\usepackage[margin=1.5in]{geometry}

\RequirePackage{mathrsfs} 
\usepackage[OT2,T1]{fontenc}
\DeclareSymbolFont{cyrletters}{OT2}{wncyr}{m}{n}
\DeclareMathSymbol{\Sha}{\mathalpha}{cyrletters}{"58}

\newtheorem{theorem}{Theorem}[section]
\newtheorem{lemma}[theorem]{Lemma}
\newtheorem{prop}[theorem]{Proposition}

\theoremstyle{definition}

\newtheorem{opm}[theorem]{Remark}

\numberwithin{equation}{section} \numberwithin{figure}{section}

\DeclareMathOperator{\Spec}{Spec}

\DeclareMathOperator{\an}{an}

\DeclareMathOperator{\PSL}{PSL}
\DeclareMathOperator{\Comm}{Comm}

\newcommand{\Qbar}{\overline{\mathbb Q}}

\newcommand{\HH}{\mathbb H}
 
\newcommand{\Fal}{\mathrm{Fal}}

\newcommand\PP{\mathbb{P}}

\newcommand\ZZ{\mathbb{Z}}

\newcommand\QQ{\mathbb{Q}}
\newcommand\RR{\mathbb{R}}
\newcommand\CC{\mathbb{C}}

\newcommand\OO{\mathcal{O}}

\title[Arakelov theory and isogenies of hyperbolic curves]{An effective Arakelov-theoretic version of the hyperbolic isogeny theorem}

\author{A. Javanpeykar}
\address{A. Javanpeykar \\
Institut f\"{u}r Mathematik\\
Johannes Gutenberg-Universit\"{a}t Mainz\\
Staudingerweg 9, 55099 Mainz\\
Germany.}
\email{peykar@uni-mainz.de}

 

\begin{document}

\begin{abstract} For any integer $e$ and hyperbolic curve $X$ over $\Qbar$, Mochizuki showed that
 there are only finitely many
 isomorphism classes of  hyperbolic curves $Y$ of Euler characteristic $e$ with the same universal
cover as $X$. We use Arakelov theory to prove an effective version of this finiteness statement.
\end{abstract}

\maketitle

\section{Introduction}
In this paper we investigate finiteness results for hyperbolic curves with  the same universal cover. 
In the terminology of
\cite{Mochi}, two hyperbolic curves have the same universal covering scheme (see \cite{Delignepi1, VW}) if they are ``isogenous", i.e., 
share a common finite \'etale cover. More precisely, 
a curve $Y$ is \emph{isogenous} to a curve $X$ 
 if there exist a curve $C$, a finite \'etale morphism $C\to X $ and and finite \'etale morphism $C\to Y$. 

If $X$ is a smooth quasi-projective connected curve over an algebraically closed field, we let  $\overline X$ be its smooth compactification. Moreover, we let $g_X$ denote the genus of $\overline X$. The  Euler characteristic $e(X)$ of $X$ is  given by $2-2g_X - r$, where $r = \# (\overline X\backslash X)$. The curve $X$ is hyperbolic if and only if $e(X) $ is negative. In particular, $X$ is non-hyperbolic if and only if $X$ is isomorphic to $\mathbb P^1$, $\mathbb A^1$, $\mathbb A^1\setminus \{0\}$, or an elliptic curve. Thus, if $X$ is a curve over $\mathbb C$, then $X$ is hyperbolic if and only if the universal covering of the Riemann surface $X(\mathbb C)$ is isomorphic to the complex upper half-plane.

Mochizuki proved that, if  $X$ is a hyperbolic  curve over $\Qbar$  and $e$ is an   integer, then 
the set of isomorphism classes of  hyperbolic curves with Euler characteristic $e$ which are isogenous to $X$ is finite.

The aim of this paper is to prove an effective version of Mochizuki's finiteness result. This means that we establish an effective bound 
on the ``height'' of a curve $Y$  in terms of  $X$ and $e(Y)$, under the assumption that $Y$ and $X$ are isogenous. To state our main result, we will need the Faltings height and the Belyi degree of a curve.

For $X$ a smooth quasi-projective connected curve over $\Qbar$, let $h_{\Fal}( X)$ be  the Faltings height of $\overline X$ (see Section \ref{section:faltings}). This 
is an intrinsic invariant of $X$ introduced by Faltings in \cite{Faltings2}. It plays a central role in the study 
of (usual) isogenies of abelian varieties and Faltings's proof of the Mordell conjecture \cite{Faltings2, Szpiroa}.

In  \cite{Belyi} Belyi proved that there exists a finite morphism $\overline X\to \PP^1_{\Qbar}$ ramified over at most three points; 
we call such a morphism a Belyi map on $ X$.  We define the Belyi degree   of $X$, denoted by $\deg_B(X)\in \ZZ_{\geq 1}$, to be the minimal degree of a Belyi map $\pi:\overline X\to \PP^1_{\Qbar}$ such that $\pi(\overline X\setminus X)\subset \{0,1,\infty\}$.

\begin{theorem}\label{thm: main theorem}
If $X$ is a hyperbolic curve over $\Qbar$ and $Y$ is isogenous to $X$, then  \[h_{\Fal}(   Y) \leq  10^{338}2^{14\vert e(X)\vert +14 \vert e(Y)\vert}  \left( e(X) e(Y)  \deg_B(X)\right)^6. \]
\end{theorem}
Mochizuki's (geometric) finiteness theorem can be considered as an analogue for hyperbolic curves of Faltings's (arithmetic) isogeny theorem for abelian varieties. Recall that Faltings proved that,
 for an abelian variety $A$ over a number field $K$, the set of $K$-isomorphism classes of abelian varieties $B$ which are $K$-isogenous to $A$ is finite. 
Faltings's theorem was made effective by Raynaud \cite{Raynaud} (following the methods of Faltings); see also the stronger results of Masser-W\"ustholz \cite{MW2,MW1} which were improved by Gaudron-R\'emond \cite{GaRe}, building on the work
of Bost-David (see \cite{Bost}).  

In light of the previous paragraph, it seems natural to ask whether Mochizuki's finiteness theorem for
hyperbolic curves over $\Qbar$ also has a natural analogue over number fields. It turns out that the corresponding
``naive'' analogue fails. Indeed, if $K$ is a number field, $g\geq 2$ is an integer and $X$ is a smooth projective connected genus $g$ curve over $K$ such that
its Jacobian $\mathrm{Jac}(X)$ has a non-trivial $K$-rational $2$-torsion point, then the set of $K$-isomorphism classes
of projective curves $Y$ over $K$ for which there exists a degree two finite \'etale $K$-morphism $Y\to X$ is infinite (and therefore the 
set of isomorphism classes of genus $2g-1$ curves isogenous to $X$ over $K$ is infinite).

We emphasize that to prove Theorem \ref{thm: main theorem}, we  combine parts of Mochizuki's original proof of \cite[Theorem A]{Mochi} (which uses the results of Margulis \cite{Margulis} and Takeuchi \cite{Takeuchi}) with a  recent result in Arakelov theory relating the Faltings height of a curve to its Belyi degree via an explicit polynomial inequality \cite{J}. In the course of making the strategy of Mochizuki effective, we will use lower bounds due to Odlyzko for root discriminants \cite{OdII} and the work of Shimizu \cite{Shimizu} and Takeuchi \cite{Takeuchi2, Takeuchi} on arithmetic curves.

\subsection*{Acknowledgements} 
We thank Robin de Jong and Rafael von K\"anel for many useful comments on the paper.
We thank Yuri Bilu,  Peter Bruin, Bas Edixhoven, Gerard Freixas i Montplet,        J\"urg Kramer,  Qing Liu and Xin Lu for many helpful and motivating discussions. Special thanks to Jean-Beno\^it Bost for giving us the idea of how to prove Lemma \ref{lem: df},  Philipp Habegger for an inspiring discussion that led to this paper and John Voight for many helpful discussions and for explaining to us several properties of arithmetic curves.
We gratefully acknowledge the support of   SFB/Transregio 45. 

\subsection{Conventions} We let $ \Qbar\subset \CC$  be the algebraic closure of $\QQ$ in $\CC$. 
Unless stated otherwise, $k$ is an algebraically closed field of characteristic zero. A curve over $k$ is a smooth quasi-projective connected one-dimensional scheme over $k$.
  If $X$ is a   curve over $k$, then we let $\overline X$ be its smooth compactification. Moreover, we let $g_X$ denote the genus of $\overline X$.  

\section{Hyperbolic Deligne-Mumford curves}\label{sect: prelim}

 In this section we aim at
giving the necessary definitions and results for what we need later (and we need at least to fix our
notation) to prove  Lemma \ref{lem: mochizuki}.

A (smooth quasi-projective connected) curve $X$ over $k$ is \emph{hyperbolic} if 
  $e(X) := 2-2g_X -r$ is negative, where $r$ is the number of points in $\overline X\backslash X$. 
Note that  a non-hyperbolic curve over $k$ is isomorphic to either an elliptic curve, 
 $\mathbb P^1_k$, $\mathbb A^1_k$ or $\mathbb G_{m,k}$.
 Also, if $k=\CC$, the universal covering space of (the analytic space associated to) a 
 hyperbolic curve   is the complex upper half-plane $\mathbb H$.

 A precise definition of a (smooth separated) Deligne-Mumford analytic stack over the category of complex manifolds is given in \cite[Section 3.2]{BN}. As is explained in \cite[Remark 3.4]{BN}, we do not gain any more generality from working over the larger category of complex analytic spaces   if we restrict ourselves to \emph{smooth} separated stacks.  Note that a  Deligne-Mumford analytic stack $X$ is (representable by) a complex manifold if and only if all the stabilizers of $X$ are trivial. If $X$ is a smooth separated Deligne-Mumford algebraic stack over $\mathbb C$, we let $X^{\mathrm{an}}$ denote the associated Deligne-Mumford analytic stack (see \cite[Section 3.3]{BN}). We will   use Noohi's stacky Riemann existence theorem   to relate finite \'etale covers of $X^{\an}$ with finite \'etale covers of $X$ (see \cite[Theorem 20.4]{Noohi}).

A smooth  separated connected Deligne-Mumford (algebraic) stack over $k$  is a \emph{Deligne-Mumford curve}
 if it is of dimension one.  By \cite{BN}, the universal covering space of the analytification $X^{\mathrm{an}}$
of a Deligne-Mumford curve $X$ over $\CC$ is either isomorphic to the complex upper half-plane $\HH$, 
the complex line $\CC$, or a weighted projective line $\mathcal P (n,m)$ of type $(n,m)$.

Let $H$ be a Deligne-Mumford curve over $k$ (with $k$ an arbitrary algebraically closed field of characteristic zero). 
Then $H$ is \emph{hyperbolic} if there exists a finite \'etale morphism $X\to H$ with $X$ a hyperbolic curve over $k$. 
Note that, by  \cite[Theorem 1.1]{BN},   a Deligne-Mumford curve $H$ 
over $\CC$ is hyperbolic if and only if the universal covering
 space of the analytic stack associated to $H$ is $\HH$.

Suppose that $H$ is generically a scheme, i.e., 
the stabilizer at the generic point of $H$ is trivial.
 Some authors refer to such Deligne-Mumford curves as 
connected (reduced) orbifold curves.
 Let $H\to H_c$ be the coarse moduli space of $H$.
 Note that $H_c$ is a smooth curve over $k$. Let $g_H$
 be the genus of the smooth compactification of $H_c$, 
and let $r_H$ be the number of points in the boundary of
 $H_c$. Let $\Sigma_H$ be the set of points of $H_c$ over which  $H\to H_c$ is not \'etale. 
For $\sigma$ in $\Sigma_H$, let $i_\sigma$ be the ramification index of $H\to H_c$ at $\sigma$.   
We define the \emph{(compactly supported) Euler characteristic of $H$} to be the rational number \[e(H):= -2g_H +2 - r_H - \sum_{\sigma \in \Sigma_H} \frac{i_\sigma - 1}{i_\sigma}.\]

\begin{lemma}\label{lem: euler}
If $H$ is a   hyperbolic  Deligne-Mumford curve over $k$ which is generically a scheme, then \[\vert e(H)\vert  \geq \frac{1}{42}.\]
\end{lemma}
\begin{proof} We follow the proof of \cite[Lemma 5.4]{Mochi}. Firstly,  by definition,
 there exists a  finite \'etale morphism $X\to H$ with $X$ a (smooth quasi-projective connected) hyperbolic curve over $k$.
 Let $d$ be the degree of $X\to H_c$. Let $g$ be the genus of the smooth compactification
 $\overline X$ of $X$, and let $r = \# \overline X\setminus X$. Let $\overline H_c$ be the
smooth compactification of $H_c$.
The Riemann-Hurwitz formula for $\overline X \to \overline{H_c}$ implies that 
\[2g-2+r =  d\left(2g_H + 2  -r_H - \sum_{\sigma \in \Sigma_H} \frac{i_\sigma - 1}{i_\sigma}\right).\] As $X$ is hyperbolic, the inequality $2g-2+r > 0$ holds. 
Therefore, $\vert e(H) \vert = 2g_H + 2  -r_H - \sum_{\sigma \in \Sigma_H} \frac{i_\sigma - 1}{i_\sigma} > 0$. 
We now proceed by a case-by-case analysis. In fact, if $2g_H - 2+ r_H\geq 1$, then $\vert e(H) \vert\geq 1 \geq \frac{1}{42}$.
Also, if $2g_H - 2 + r_H = 0$, we have $\vert e(H) \vert \geq \frac{1}{2}\geq \frac{1}{42}$. Moreover, if 
$2g_H - 2 + r_H = -1$, then $\vert e(H) \vert \geq \frac{1}{6}\geq \frac{1}{42}$. Therefore, to prove the lemma, we may assume that $2g_H - 2+r_H =-2$, in which case
$g_H = r_H =0$. Now, if $\Sigma_H$ has at least five elements, the inequalities $\vert e(H)\vert  \geq \frac{1}{2}\geq \frac{1}{42}$ hold. Also, if $\Sigma_H$ has precisely four 
elements, then $\vert e(H)\vert \geq \frac{1}{2}\geq \frac{1}{42}$. Finally, to prove the lemma, we can assume that $\Sigma_H$ has precisely three elements, say $\sigma_1$, $\sigma_ 2$ and $\sigma_3$. 
Note that it is never the case that two $i_\sigma$'s are equal to two. Therefore, reordering the elements of $\Sigma$, we can assume that $2\leq i_{\sigma_1} < i_{\sigma_2} \leq i_{\sigma_3}$. If $i_{\sigma_3} \geq 7$, then 
\[e(H) \geq 1-\frac{1}{2} -\frac{1}{3} -\frac{1}{7} \geq \frac{1}{42}.\] If $i_{\sigma_3} \leq 6$, then 
\[(i_{\sigma_1}, i_{\sigma_2}, i_{\sigma_3}) \in \{(2,4,6),(2,4,5),(3,3,6),(3,3,5),(3,3,4)\}.\] 
Note that in each case $\vert e(H)\vert  \geq \frac{1}{42}$. The lemma is proven.
\end{proof}

\section{The hyperbolic core}\label{section: hyp} In this section, we follow \cite[Section 2]{Mochi} and introduce the hyperbolic core of a hyperbolic curve.

 Let $X$ be a (smooth quasi-projective connected) hyperbolic curve over $\CC$ with universal covering 
$\HH  \to X(\CC)$.  Let $\Gamma_X\subset \PSL_2(\RR)$ be the
  image of $\pi_1(X) \to \mathrm{Aut}(\HH) = \PSL_2(\RR)$.  
 If $\Gamma_1$ and $\Gamma_2$ are subgroups of
  $\PSL_2(\RR)$, we will write $\Gamma_1 \sim \Gamma_2$ if $\Gamma_1$ and $\Gamma_2$
 are commensurable in $\PSL_2(\RR)$. That is, $\Gamma_1\sim\Gamma_2$ 
if and only if $\Gamma_1\cap \Gamma_2$ has finite index in both $\Gamma_1$ and $\Gamma_2$. We now consider \[\Comm(\Gamma_X) := \{\gamma \in \PSL_2(\RR) \ | \ (\gamma \cdot \Gamma_X \cdot \gamma^{-1})\sim \Gamma_X \}.\] Note that $\Gamma_X\subset \Comm(\Gamma_X)$.

We will say that a hyperbolic curve $X$ is
 \emph{arithmetic} if the (torsion-free) subgroup $\Gamma_X\subset \PSL_2(\RR)$ is of infinite index in $\mathrm{Comm}(\Gamma_X)$.
 Margulis proved that $X$ is arithmetic if and only if $\Gamma_X\subset \PSL_2(\RR)$ is an
 arithmetic subgroup; see   \cite{Margulis}. For basic facts on arithmetic groups  
we refer to Shimura \cite{Shimura} (see also \cite[Section 2]{Mochi}).  
Note that all  arithmetic curves are hyperbolic (with this definition).

Let $X$ be a  hyperbolic non-arithmetic curve over $\CC$. 
Note that $\HH$ is the universal covering of $X$ (in the analytic topology).  
The quotient stack $[\Comm(X) \backslash \HH]$ is the analytification of a Deligne-Mumford curve $H_X$ over $\CC$ which is
generically a scheme. The inclusion $\Gamma_X \subset \Comm(X)$ induces
 a finite \'etale morphism $X\to H_X$ of Deligne-Mumford curves.  
We follow \cite{Mochi} and call $H_X$ the \emph{hyperbolic core} of $X$.

If $k$ is an algebraically closed field of characteristic zero, we extend the notion of ``arithmeticity'' to  hyperbolic curves over $k$ in the
obvious way. More precisely, we say that a hyperbolic curve $X$ is  \emph{arithmetic} if there exist a subfield $k_0$ of $\CC$ and a morphism $k_0\to k$ 
such that $X$ has a model $X_0$ over $k_0$ with the property
 that $X_{0,\CC}$ is arithmetic.  Note that if $X$ is (hyperbolic and) arithmetic, then for any choice of $k_0$, $X_0$ and $k_0\to k$ the curve $X_{0,\CC}$ is arithmetic (as the isomorphism class of $X_{0,\CC}$ is   determined by the isomorphism class of $X$ over $k$). Moreover, any (smooth quasi-projective connected) curve
isogenous to $X$ is hyperbolic and arithmetic.

\section{Isogenies of non-arithmetic curves}
  
Let $X$ and $Y$ be curves over $k$. A pair $(C\to X, C\to Y)$ is an \emph{isogeny} from $X$ to $Y$ if 
$C\to X$ is finite \'etale and $C\to Y$ is finite \'etale. We say that $X$ and $Y$ are \emph{isogenous} 
if there is an isogeny from $X$ to $Y$.

\begin{lemma}\label{lem: mochizuki} If $X$ is a hyperbolic non-arithmetic curve over $k$ and $Y$ is a curve isogenous to $X$,  then 
there exist a smooth quasi-projective connected curve $C$ over $k$ and finite \'etale morphisms 
\[ \xymatrix{ & C \ar[dl]_{\pi_X} \ar[dr]^{\pi_Y} & \\ X & & Y}\] 
such that \[ \deg \pi_X \leq 42 \vert e(Y)\vert , \quad \deg \pi_Y \leq 42 \vert e(X)\vert  .\] 
\end{lemma}
\begin{proof}
We follow \cite{Mochi}.
 Firstly, we may and do assume 
that $k = \CC$. Then, as $X$ is hyperbolic and non-arithmetic, 
the hyperbolic core $H_X$ of $X$ is a well-defined Deligne-Mumford curve with trivial generic stabilizer.
 By \cite[Proposition 3.2]{Mochi}, we have an inclusion
 $\Gamma_Y \subset \Comm(\Gamma_X)$.   Define $\Gamma_C := \Gamma_X \cap \Gamma_Y$, where 
the intersection is taken in $\mathrm{Comm}(X)$. Let $\pi_X:C\to X$ and $\pi_Y:C\to Y$
 be the associated finite \'etale morphisms, where $C = \Gamma_C\backslash \HH$.  Note that $\deg \pi_X$
 is at most the index $[\Comm(X) :\Gamma_Y]$ of $\Gamma_Y$ in $\mathrm{Comm}(X)$. Similarly, $\deg \pi_Y  = [\Gamma_Y:\Gamma_C ]\leq [\Comm(X):\Gamma_X]$.
  The  Riemann-Hurwitz formula implies that \[e(Y) = [\Comm(X) : \Gamma_Y] e(H_X), \quad e(X) = [\Comm(X) :\Gamma_X] e(H_X).\] Therefore, by Lemma \ref{lem: euler}, 
\begin{eqnarray*} \vert e(Y)\vert &=&  [\Comm(X) : \Gamma_Y] \vert e(H_X) \vert\geq \frac{ [\Comm(X) : \Gamma_Y]}{42} \geq \frac{\deg \pi_X}{42}, \\
\vert   e(X)\vert  &=&  [\Comm(X) :\Gamma_X] \vert e(H_X)\vert \geq \frac{\deg \pi_Y}{42}. 
\end{eqnarray*}
 This proves the lemma.
\end{proof}

\section{Isogenies of arithmetic curves} 
In the previous section we showed (following the methods of Mochizuki \cite{Mochi}) how to find an isogeny of bounded degree between two isogenous non-arithmetic curves. In this section we extend these results to arithmetic curves.  

To prove that there exists an isogeny of bounded degree betweeen two isogenous arithmetic curves, we use standard arguments due to Takeuchi (which are now extensively used in the literature \cite{Sijsling1,  Voight}). To be more precise,  we will use Takeuchi's methods \cite{Takeuchi2, Takeuchi} to construct explicit isogenies, and then bound the degree of this isogeny using Odlyzko's bounds for root discriminants of number fields \cite{OdII} and Shimizu's mass formula \cite{Shimizu}. These methods were  used by Takeuchi (see \cite{Takeuchi}) to prove his finiteness result  (stated by Mochizuki as \cite[Theorem 2.6]{Mochi}) for arithmetic curves of fixed Euler characteristic.

\begin{lemma}\label{lem: affar}
Let $X$ be an  smooth quasi-projective connected   arithmetic curve over  $k$ and let $g_X$ be the genus of its  smooth compactification $\overline X$.  
\begin{enumerate}
\item  If $X$ is affine, there exist a smooth projective connected curve $\overline C$ over $k$, 
a Belyi map $\pi_1: \overline C\to \PP^1_k$  
 and a finite morphism $\pi_2: \overline C\to \overline  X$ such that \[\deg \pi_1 \leq  10^{46} 2^{2g_X+3} \vert e(X) \vert, \quad \deg \pi_2 \leq 2^{2g_X+2}.\]
\item If $X$ is projective and $Y$ is a curve isogenous to $X$, then there exist a smooth projective connected
curve $\widetilde C$ over $k$ and finite \'etale morphisms 
\[ \xymatrix{ & \widetilde C \ar[dl]_{\pi_X} \ar[dr]^{\pi_Y} & \\ X & & Y}\] 
such that 
 \begin{align*}
\deg \pi_X & \leq  10^{48} 2^{2g_X +2g_Y} (2g_Y-2) \\
\deg \pi_Y & \leq 10^{48} 2^{2g_X + 2g_Y } (2g_X-2).
\end{align*}
\end{enumerate}
\end{lemma}
\begin{proof}  
To prove the result, we may and do assume that $k=\CC$.  Let $X$ be an   arithmetic curve over $\CC$. We start with some general observations due to Takeuchi \cite{Takeuchi2}.

By assumption, the subgroup $ \Gamma_X\subset \PSL_2(\RR)$ is arithmetic (see 
Section \ref{section: hyp} for the notation).  
Define $\Gamma^{(2)}_X\subset \Gamma_X$ to be the subgroup generated by $\{\delta^2 \ | \ \delta \in \Gamma_X\}$. 
 In \cite[Section 3]{Takeuchi2} Takeuchi proves that the field  $F$ 
obtained by adjoining the elements
$\mathrm{tr}(\gamma)$ with $\gamma \in \Gamma^{(2)}_X$
to $\QQ$ is a totally real number field. Also,
there exist  a quaternion algebra $A$ over $F$ 
which is trivial at only one  of the infinite places of $F$ and
a trivialization of $A$ at the trivial infinite place such that 
the image of $A\otimes \RR$ in $\mathrm{M}_2(\RR)$ is the subalgebra 
$F[\Gamma^{(2)}_X]$. Furthermore, as
$\OO_F[\Gamma^{(2)}_X]$ is an order of $A$ (with $\OO_F$ the ring of integers of $F$), there exists a maximal order $\mathcal O$ 
of $A$ such that $\Gamma^{(2)}_X$ is a subgroup of the image of the 
group $\Gamma(\mathcal O, 1)$ in $\mathrm{M}_2(\RR)$,  where
\[\Gamma(\mathcal O, 1) = \{x\in \mathcal O \ | \ \mathrm{nrd}(x) = 1 \} \] and $\mathrm{nrd}$
is the reduced norm of $A$ over $F$.    Let $C$ be a complex algebraic curve with analytification $\Gamma^{(2)}_X\backslash \HH$. As $\Gamma_X^{(2)} \subset \Gamma_X$, there is a finite \'etale morphism $q:C\to X$. Similarly, let  $X(F,A,\mathcal O)$ be  a complex algebraic curve  with analytification
 $\Gamma(\mathcal O,1)\backslash \HH$. Then, the canonical inclusion  $\Gamma_X^{(2)}\subset \Gamma(\mathcal O, 1)$ induces a finite \'etale morphism $p:C\to X(F,A,\mathcal O)$. We obtain the following diagram of finite \'etale morphisms:

 \[ \xymatrix{ & C \ar[dl]_{p} \ar[dr]^q & \\ X(F,A,\mathcal O) & & X.}\] 

Note that $\deg q $ is bounded from above by the index of $\Gamma^{(2)}_X$ in $\Gamma_X$.
As the index of $\Gamma^{(2)}_X$ in $\Gamma_X$ is at most
 $4[\{\pm 1\} \Gamma_X:\{\pm 1\} \Gamma^{(2)}]$, by \cite[Proposition 2.2.(i)]{Takeuchi}, 
 the index of $\Gamma^{(2)}_X$ in $\Gamma_X$ is at most $4\cdot 2^{2g_X }$. We conclude that   $\deg q \leq 2^{2g_X+2}$.  

Now,  following the arguments of Takeuchi \cite[p. 383-384]{Takeuchi}, we will show that 
\begin{align}\label{p}
 \deg p & \leq 10^{46} 2^{2g_X +3} \vert e(X)\vert.
 \end{align} To do so, 
we   use Shimizu's formula \cite{Shimizu}. Write $n= [F:\QQ]$ and $d_F$ for the absolute value of the discriminant of $F$. Also, let $D_A$ be the discriminant of the quaternion algebra $A$ over $F$ (which is defined to be the product of all maximal ideals $\mathfrak p$ of $\OO_F$ such that $A$ ramifies at $\mathfrak p$). Furthermore, let $\zeta_F$ be the Dedekind zeta function of the number field $F$. Since $X(F,A,\mathcal O)$ is  the quotient of $\mathbb H$ by a cofinite Fuchsian group, the volume of the Riemann surface $\mathrm{Vol}(X(F,A,\OO))$ (induced by the standard volume form on $\mathbb H$) is a well-defined real number. As in \cite[Equation 2.4]{Takeuchi}, Shimizu's formula gives an expression for the volume of $X(F,A,\mathcal O)$.  which reads \[ \mathrm{Vol}(X(F,A,\OO)) = \frac{4   d_F^{3/2} \zeta_F(2) \prod_{\mathfrak p\vert D_A} (N_{F/\QQ}(\mathfrak p) - 1)}{(2\pi)^{2n}}. \] Then, as in \cite[Equation 2.5]{Takeuchi}, we use Riemann-Hurwitz and the diagram above to see that 
\begin{align}\label{sh}
\frac{4 d_2  d_F^{3/2} \zeta_F(2) \prod_{\mathfrak p\vert D_A} (N_{F/\QQ}(\mathfrak p) - 1)}{(2\pi)^{2n}} & = d_1\vert e(X)\vert , 
\end{align} where $d_1 = [\Gamma_X\cdot \{\pm 1\} : \Gamma_X^{(2)}\cdot \{\pm 1\} ]$ and $d_2 = [\Gamma(\OO,1) : \Gamma_X^{(2)}\cdot \{\pm 1\} ]$. Note that $d_1 \leq \deg q \leq 2^{2g_X + 2}$ and that $\deg p \leq 2 d_2$.  Since $\zeta_F(2) \geq 1$ and $\prod_{\mathfrak p \vert D_A} (N_{F/\QQ} (\mathfrak p) - 1) \geq 1$, formula (\ref{sh}) implies that

 \[d_2 \frac{4 d_F^{3/2}}{(2\pi)^{2n}} \leq \frac{4 d_2  d_F^{3/2} \zeta_F(2) \prod_{\mathfrak p\vert D_A} (N_{F/\QQ}(\mathfrak p) - 1)}{(2\pi)^{2n}}  = d_1\vert e(X)\vert. \]

We now invoke Odlyzko's explicit lower bound for the discriminant of a totally real number field \cite[Theorem 1.(2)]{OdII} to see that \[d_F \geq 50^n \exp(-70).\] In particular, \[ \frac{4d_F^{3/2}}{(2\pi)^{2n}} \geq\frac{ (50^n\exp(-70))^{3/2}}{(2\pi)^{2n}} \geq  \exp(-70\cdot 3/2) > 10^{-46}.\] Therefore, the sought upper bound (\ref{p}) for $\deg p$ follows from 
 \begin{align*} 
10^{-46} \deg p  & \leq   10^{-46} 2 d_2    \leq 2d_2 \frac{4d_F^{3/2}}{(2\pi)^{2n}} \leq 2d_1 \vert e(X) \vert \leq 2^{2g_X +3} \vert e(X)\vert.
\end{align*} 

We conclude that,  if $Y$ is a curve isogenous to $X$, then there exist a smooth projective connected curve $C^\prime$, a finite \'etale morphism $p^\prime:C^\prime\to X(F,A,\mathcal O)$ and a finite \'etale morphism $q^\prime:C^\prime\to Y$ such that \[\deg p^\prime \leq   10^{46} 2^{2g_Y +3} \vert e(Y)\vert   \] and \[\deg q^\prime \leq 2^{2g_Y+2}. \] We now use this to prove the lemma. 

If $X$ is affine, then
the curve $X(F,A,\mathcal O)$ is affine. Therefore, by \cite{MP}, up to conjugation by the action
of $\mathrm{GL}_2(\RR)$, 
 we have that $F = \QQ$ and $A=\mathrm{M}_2(\QQ)$. In particular, $X(F,A,\mathcal O) = Y(2)$. 
Let $Y(2) \to \PP^1 \setminus \{0,1,\infty\}$ be the isomorphism induced by the $\lambda$-function
on $\HH$. This extends to an isomorphism $\overline{X(F,A,\mathcal O)} = X(2)\to \PP^1$.  Let $\overline C$ be the smooth compactification of the smooth quasi-projective connected algebraic curve (whose analytification is) $\Gamma^{(2)}\backslash \HH$,
and let $\pi_2:\overline C\to \overline X$ be the natural extension of the 
finite morphism $q:C\to X$. Note that $\deg \pi_2 = \deg q \leq 2^{2g_X +2}$. Furthermore, the composed morphism $\pi_1:\overline C\to X(2)\to \PP^1$ is a Belyi map of
 degree at most  $\deg \pi_1 =\deg p \leq 10^{46} 2^{2g_X+3} \vert e(X)\vert $.  This concludes the proof of the first part  of the lemma.

Finally, to prove the second statement, let us assume that $X$ is projective and that $Y$ is a curve isogenous to $X$.  As explained above, we then have a diagram  \[ \xymatrix{ & C \ar[dr]^{p} \ar[dl]_{q} & &  C^\prime \ar[dl]_{p^\prime} \ar[dr]^{q^\prime} & \\ X & & X(F,A,\mathcal O) & & Y.}\] In particular, if  $\widetilde C $ is a connected component of the smooth quasi-projective one-dimensional scheme $ C^\prime \times _{X(F,A,\OO)} C$, then we get

\[\xymatrix{& & & &  \widetilde C \ar[ddllll]_{\pi_X} \ar[ddrrrr]^{\pi_Y} \ar[dl]^{\pi_C} \ar[dr]_{\pi_{C^\prime}} & & & & \\  & & & C \ar[dl]^q \ar[dr]^p & & C^\prime  \ar[dl]_{p^\prime} \ar[dr]_{q^\prime} &  &  & & \\ X &=   & X &  & X(F,A,\OO) & &   Y &    = & Y }\]  where 
\begin{align*}
\deg \pi_X & = \deg q \deg \pi_C \leq \deg q \deg p^\prime \leq   2^{2g_X+2} \cdot  10^{46} 2^{2g_Y +3} \vert e(Y)\vert   \\
\deg \pi_Y & = \deg q^\prime \deg \pi_{C^\prime} \leq \deg q^\prime \deg p \leq  2^{2g_Y+2}  \cdot  10^{46} 2^{2g_X +3} \vert e(X) \vert.
\end{align*} This concludes the proof of the lemma. 
\end{proof}

\section{The Faltings height and the Belyi degree}\label{section:faltings}

Let $K$ be a number field with ring of integers $\OO_K$. 
Recall that a metrised line bundle $(\mathcal{L},\|{\cdot}\|)$ on $\Spec \OO_K$ corresponds to an invertible $\OO_K$-module $\mathcal L$ with hermitian metrics on the $\mathcal L_\sigma:=\CC\otimes_{\sigma,\OO_K} \mathcal L$. The \emph{Arakelov degree} of~$(\mathcal{L},\|{\cdot}\|)$ is the real number defined by:
\begin{eqnarray*}\label{eqn_ar_degree}
\widehat{\deg}(\mathcal{L})= \widehat{\deg}(\mathcal{L},\|{\cdot}\|) =
\log\#\frac{\mathcal L}{\OO_K\cdot s} -\sum_{\sigma\colon K\to\CC}\log\|s\|_\sigma,
\end{eqnarray*}
where $s$ is any non-zero element of~$L$. Note that this is a well-defined real number attached to $(\mathcal L, \| {\cdot}\|)$, as independence of the choice of~$s$ follows from the product formula. 

Let $A$ be a $g$-dimensional abelian variety over $K$ and let $\mathcal A$ be the N\'eron model of $A$ over $\OO_K$. Then we have the locally free $\OO_K$-module  $\mathrm{Cot}_0 (A) :=  0^\ast \Omega_{\mathcal A/\OO_K}$  of rank $g$, and hence the invertible  $\OO_K$-module of rank one:\[ \omega_A:= \Lambda^g \mathrm{Cot}_0(A).\] For each complex embedding $\sigma:K\to \CC$, we have the scalar product on $\CC \otimes_{\OO_K}\omega_{A}$ given by  \[( \omega,\eta)= \frac{i}{2}(-1)^{g(g-1)/2}\int_{A_\sigma(\CC)} \omega\overline{\eta}. \] 
The \textit{relative Faltings height}  of $A$ over $K$ is then defined to be the Arakelov degree of the metrized line bundle $\omega_A$, \[ h_{\Fal}(A/K) =  \widehat{\deg} \  \omega_A.\]  
To define the ``stable Faltings height'' of $A$, let $L/K$ be a finite field extension such that $A_L$ has semi-stable reduction over $\OO_L$. Then the \textit{stable Faltings height} of $A$ is defined to be \[ h_{\Fal,\textrm{s}}(A) := \frac{h_{\Fal}(A_L/L)}{[L:K]}.\] 
 
We  now define the Faltings height as a real-valued function on the set of $\Qbar$-isomorphism classes of abelian varieties over $\Qbar$.
Let $A$ be an abelian variety over $\Qbar$.
 Let $K$ be a number field such that the abelian
 variety $A$ has a model $A_0$ over $K$ with 
semi-stable reduction over $\OO_K$. Then the 
\textit{Faltings~height} of $A$ is defined as~$h_{\Fal}(A) := h_{\Fal,\textrm{s}}(A_0)$. This is a 
well-defined real number attached to the isomorphism class 
of $A$ over $\Qbar$; see \cite[Section 5.4]{Moret-Bailly3}.
  
Let $X$ be a smooth projective connected curve over $\Qbar$. 
We define the Faltings height $h_{\Fal}(X)$ of $X$ to be the Faltings height of its Jacobian $\mathrm{Jac}(X) = \mathrm{Pic}^0_{X/\Qbar}$. If the genus of $X$ is positive, one can also compute the Faltings height of $X$ without passing to the Jacobian as follows. 
Let $K$ be a number field such that the minimal regular model $p:\mathcal X\to \Spec \OO_K$ of $X$ over $\OO_K$ is semi-stable.
 We endow $\det p_\ast \omega_{\mathcal X/\OO_K}$ with the Faltings metric \cite{Faltings1}.
 Then  \[h_{\Fal}(X) = \frac{\widehat{\deg} \det p_\ast \omega_{\mathcal X/\OO_K}}{[K:\QQ]};\] see  \cite[Lemme 3.2.1, Chapitre 1]{Szpiroa}.

The next result  implies  an Arakelov-theoretic effective version of the classical theorem of De Franchis-Severi which states that,  for $X$ a smooth projective connected curve, the set of isomorphism classes of projective hyperbolic curves $Y$ which are dominated by $X$ is finite. This result is not new; for lack of reference we include a proof.

\begin{lemma}\label{lem: df} If $f:X\to Y$ is  a finite morphism of projective curves over $\Qbar$, then
 \[  h_{\Fal}(Y)\leq h_{\Fal}(X) + g_X \log(2\pi \deg f).\] 

\end{lemma}
\begin{proof}  
Note that $f^\ast :\mathrm{Jac}(Y)\to \mathrm{Jac}(X)$ and $f_\ast:\mathrm{Jac}(X)\to \mathrm{Jac}(Y)$ satisfy $f_\ast f^\ast = [\deg f]$, where $[\deg f]:\mathrm{Jac}(Y)\to\mathrm{Jac}(Y)$ is   multiplication by $\deg f$   on $\mathrm{Jac}(Y)$. Let $K$ be the kernel of $f_\ast$ and note that $K$ is a smooth proper group scheme over $\Qbar$ whose identity component $K^0$ is an abelian subvariety of $\mathrm{Jac}(X)$.   Let $$\phi:\mathrm{Jac}(Y)\times K^0\to \mathrm{Jac}(X) $$ be the morphism of abelian varieties given by $\phi(P,Q) = f^\ast (P) + Q$.
It is not hard to show that $\phi$ is an isogeny of degree at most $(\deg f)^{2g_Y}$. Indeed,  as $f_\ast$ is surjective, \[\dim K^0 = \dim K = g_X-g_Y.\] In particular, to show that $\phi$ is an isogeny of degree at most $(\deg f)^{2g_Y}$ it suffices to show its kernel has at most $(\deg f)^{2g_Y}$ elements. To do so, note that 
\begin{align*}
\ker(\phi) & =  \{(P,Q) \in \mathrm{Jac}(Y) \times K^0 \ | \ f^\ast (P) = -Q\} \\
& \cong \{ P \in \mathrm{Jac}(Y) \ | \ f^\ast (P) \in K^0\} \\
& \subset \{P \in \mathrm{Jac}(Y) \ | \ f^\ast (P) \in K\} \\
& = \ker ([\deg f] : \mathrm{Jac}(Y)\to \mathrm{Jac}(Y)),
\end{align*}  where we used that $f_\ast f^\ast = [\deg f]$ to obtain the last equality. Now,
since $\phi$ is an isogeny of degree at most $(\deg f)^{2g_Y}$, we have
\begin{align*}
 h_{\Fal}(\mathrm{Jac}(Y) \times K^0) & \leq h_{\Fal}(\mathrm{Jac}(X)) + \frac{1}{2}\log(\deg \phi)  \\
& \leq h_{\Fal}(X) + g_Y \log(\deg f), 
\end{align*} where we used that, if $\psi:A\to B$ is an isogeny of abelian varieties over $\Qbar$, then $h_{\Fal}(A) \leq h_{\Fal}(B) + \frac{1}{2}\log(\deg \psi)$; see \cite[Equation 2.6]{Bost}. 
By the additivity of the Faltings height  (see \cite[Equation 2.7]{Bost}), the inequality 
\begin{align*}
 h_{\Fal}(Y) + h_{\Fal}(K^0)  & = h_{\Fal}(\mathrm{Jac}(Y)) + h_{\Fal}(K^0) = h_{\Fal}(\mathrm{Jac}(Y) \times K^0)  \\
& \leq h_{\Fal}(\mathrm{Jac}(X)) + g_Y \log(\deg f) \\
&= h_{\Fal}(X) + g_Y \log(\deg f)
\end{align*} holds.
Finally, we use that $\dim K^0 = g_X - g_Y$ and apply the lower bound of Bost $h_{\Fal}(K^0) \geq -\frac{1}{2}\log(2\pi)(g_X- g_Y)$ (see  \cite[Corollaire 8.4]{GaRe}) to conclude that
\begin{align*}\label{1}
h_{\Fal}(Y) &    \leq h_{\Fal}(X) - h_{\Fal}(K^0) + g_Y \log(\deg f) \\ 
& \leq h_{\Fal}(X) + \frac{1}{2}(g_X - g_Y) \log(2\pi) + g_Y \log(\deg f).  
\end{align*} This concludes the proof of the lemma.
 \end{proof}

A finite morphism $X\to \PP^1_{\Qbar}$ is a \emph{Belyi map (on $X$)} if it ramifies over at most three points. In \cite{Belyi} 
Belyi showed that there exists a Belyi map on $X$.
Recall that the Belyi degree $\deg_B(X)$ of $X$ is  the minimal degree of a Belyi map $X\to \PP^1_{\Qbar}$ on $X$.
Important to us is the fact that the set of isomorphism classes of smooth projective connected curves of bounded Belyi degree is finite.

\begin{lemma}\label{lem: bel}
For all real numbers $C$, the set of isomorphism classes of projective curves $X$ over $\Qbar$ such that $\deg_B(X) \leq C$ is finite.
\end{lemma} \begin{proof}
 As the topological fundamental group of $\PP^1(\CC) -\{0,1,\infty\}$ is finitely generated, for all integers $d$, the set of 
isomorphism classes of degree $d$ finite \'etale covers of $\PP^1(\CC) -\{0,1,\infty\}$ is finite.  Since  the base-change functor from $\Qbar$ to $\CC$ from the category of finite \'etale covers of $\PP^1_{\Qbar}-\{0,1,\infty\}$ to the category of
finite \'etale covers of $\PP^1_{\CC}-\{0,1,\infty\}$ is an equivalence, this concludes the proof of the lemma.
\end{proof}

The next result in Arakelov theory gives an effective version of Lemma \ref{lem: bel}.  
Its proof relies   on   arithmetic properties of the Weierstrass divisor \cite[Theorem 5.9]{deJo},   Lenstra's 
generalization of Dedekind's discriminant theorem \cite[Proposition 4.1.1]{J}, sup-norm bounds
for the hyperbolic and the Arakelov metric \cite{JorKra1}, and a theorem of Merkl on Arakelov-Green's functions \cite{Merkl}.

\begin{lemma}\label{lem: j} If $X$ is a  smooth projective connected curve over $\Qbar$, then \[h_{\Fal}(  X) \leq  10^9  \deg_B(X)^6.\]
\end{lemma}
\begin{proof} Let $g$ be the genus of $X$.
By \cite[Theorem 1.1.1]{J}, the inequality \[h_{\Fal}(X) \leq  10^9  g \deg_B(X)^5.\] holds. Note that Riemann-Hurwitz implies   $g\leq \deg_B(X)$. This proves the lemma.
\end{proof}

\begin{opm}
Results similar to  Lemma 
 \ref{lem: j} (and its analogue for the ``self-intersection of the dualizing sheaf $e(X)$'') 
were proven for certain modular curves and Fermat curves in \cite{AbUl},  \cite{EdJo3}, \cite{JorKra2},  \cite{Mayer}, \cite{MicUll}, and    \cite{Ullmo}. 
\end{opm}

\begin{opm}
In \cite{BiSt} Bilu and Strambi also prove an effective   version
of Lemma \ref{lem: bel} in which they bound the ``naive" height of the curve $X$ in terms of the Belyi degree.
\end{opm}

\section{Proof of Theorem \ref{thm: main theorem}}
We start by proving bounds for heights of curves isogenous to an affine arithmetic curve.

\begin{prop}\label{prop:aff1}
If $X$ is a smooth affine connected arithmetic curve over $\Qbar$  with smooth compactification $\overline X$, then \[h_{\Fal}( X) \leq  10^{300} 2^{14 \vert e(X)\vert } e(X) ^6\]
\end{prop}

\begin{proof}  By Lemma \ref{lem: affar},  there exist a smooth projective connected curve $ C$ over $\Qbar$, 
a Belyi map $\pi_1:   C\to \PP^1_{\Qbar}$  
 and a finite morphism $\pi_2:  C\to \overline X$ such that \[\deg \pi_1 \leq  10^{46} 2^{2g_X+3} \vert e(X) \vert, \quad \deg \pi_2 \leq 2^{2g_X+2}.\] We first bound the height of $X$ in terms of the height of $ C$. More precisely, by the effective version of De Franchis-Severi (Lemma \ref{lem: df}), 
 \[  h_{\Fal}( X) \leq h_{\Fal}( C) + g_C \log(2\pi \deg \pi_2).\]  We now use that the Faltings height can be 
bounded explicitly in terms of the Belyi degree (Lemma \ref{lem: j}) to see  that 
\[h_{\Fal}(C) \leq   10^9 \deg_B( C )^6.\] Since $\pi_1:C\to \PP^1$ is  a Belyi map and $g_C\leq \deg_B(C)$, we conclude that
\begin{align*}
h_{\Fal}( X) & \leq  h_{\Fal}( C) + g_C \log(2\pi \deg \pi_2) \\
& \leq    10^9 \deg_B( C )^6 + \deg_B(C)\log(2\pi \deg \pi_2) \\
& \leq 10^9 (\deg \pi_1)^6 + (\deg \pi_1) \log(2\pi \deg \pi_2) 
\end{align*} We now use the upper bound for $\deg \pi_2$ to see that
\begin{align*}
h_{\Fal}( X) & \leq    10^9 (\deg \pi_1)^6 + (\deg \pi_1)\log(2\pi \deg \pi_2)             \\ 
& \leq 10^9 (\deg \pi_1)^6 + (\deg \pi_1) \log(2\pi 2^{2g_X+2}) \\
& \leq 10^9(\deg \pi_1)^6 + (\deg \pi_1) 2^{2g_X+5} \\
& \leq 10^9 2^{2g_X+5} (\deg \pi_1)^6.
\end{align*} We now invoke the upper bound for $\deg \pi_1$ and obtain
\begin{align*}
h_{\Fal}( X) & \leq   10^9 2^{2g_X+5} (\deg \pi_1)^6 \leq   10^9 2^{2g_X+5} ( 10^{46} 2^{2g_X+3} \vert e(X) \vert)^6 \\
&= 10^{285} 2^{14g_X+23} \vert e(X) \vert^6.    
\end{align*} Note that the hyperbolicity of $X$ implies that $g_X \leq \vert e(X)\vert$.
As $2^{23} \leq 10^8$, this proves the proposition. 
\end{proof}

\begin{prop}\label{prop:aff}
If $X$ is a smooth affine connected arithmetic curve over $\Qbar$ and $Y$ is a smooth affine connected curve isogenous to $X$ with smooth compactification $\overline Y$, then \[h_{\Fal}( Y) \leq  10^{300} 2^{14 \vert e(Y)\vert } e(Y) ^6\]
\end{prop}
\begin{proof}
This follows from Proposition \ref{prop:aff1}.
\end{proof}
 
We are now ready to prove the main result of this paper.  

\begin{proof}[Proof of Theorem \ref{thm: main theorem}] Note that  \[10^{300} 2^{14 \vert e(Y)\vert } e(Y) ^6\leq 10^{338}2^{14\vert e(X)\vert +14 \vert e(Y)\vert}  \left( e(X) e(Y)  \deg_B(X)\right)^6.\]
Therefore, by Proposition \ref{prop:aff}, we may and do assume that $X$ is not a smooth  affine connected arithmetic curve. In particular, we have that
$X$ is either a  non-arithmetic curve or a projective arithmetic curve. 
By Lemma \ref{lem: mochizuki} and Lemma \ref{lem: affar}, there exist a smooth quasi-projective connected curve $C$ over $\Qbar$, a finite \'etale morphism $\pi_X: C\to X$ 
and a finite \'etale morphism $\pi_Y:C\to Y$ with
 \begin{align*}
\deg \pi_X & \leq  10^{48} 2^{2g_X +2g_Y} (2g_Y-2) \\
\deg \pi_Y & \leq 10^{48} 2^{2g_X + 2g_Y } (2g_X-2).
\end{align*} Here we used that the bounds in Lemma \ref{lem: affar} are     bigger than the bounds in Lemma \ref{lem: mochizuki}.
  Since $g_C \leq \deg_B(\overline C)$ and $\deg_B(\overline C) \leq \deg_B(C)$, invoking  Lemma \ref{lem: df} we see that \[h_{\Fal}( Y) \leq h_{\Fal}( C) + \deg_B(C) \log(2\pi \deg \pi_Y).  \] Note that by Lemma \ref{lem: j} the Faltings height can be 
bounded explicitly in terms of the Belyi degree of the (not necessarily projective) curve $C$. Namely, \[h_{\Fal}( C) \leq 10^9 \deg_B(\overline C )^6 \leq 10^9 \deg_B(C)^6.\] Combining these inequalities, we obtain 
\[h_{\Fal}( Y) \leq   10^9 \deg_B(C)^6 + \deg_B(C)\log(2\pi \deg \pi_Y).\] Invoking the upper bound for $\deg \pi_Y$ this gives
\begin{align*}
h_{\Fal}( Y) & \leq 10^9 \deg_B(C)^6 + \deg_B(C) \log (2\pi  10^{48} 2^{2g_X + 2g_Y } (2g_X-2)) \\
&\leq 10^9 \deg_B(C)^6 + \deg_B(C) \log (  10^{49} 2^{2g_X + 2g_Y } (2g_X-2)) \\
& \leq 10^9 \deg_B(C) + 10^{49}\deg_B(C)^6 2^{2g_X+ 2g_Y}(2g_X -2) \\
& \leq 10^{50}2^{2g_X+2g_Y}(2g_X-2)\deg_B(C)^6 \\
&\leq  10^{50} 2^{2\vert e(X) \vert +2\vert e(Y) \vert}\vert e(X)\vert \deg_B(C)^6.
\end{align*}
Finally, we use basic properties of the Belyi degree to bound $\deg_B(C)$. Namely, since $\pi_X:C\to X$ is finite \'etale, it is easy to see that 
\[\deg_B(C) \leq \deg (\pi_X) \deg_B(X).\] This inequality and the above upper bound for $\deg (\pi_X)$ imply that 
\begin{align*} 
\deg_B(C) &\leq  10^{48} 2^{2g_X +2g_Y} (2g_Y-2) \deg_B(X) \\  &\leq 10^{48} 2^{2\vert e(X) \vert +2\vert e(Y)\vert } \vert e(Y)\vert \deg_B(X) .\end{align*}
We conclude that 
\begin{align*} h_{\Fal}( Y) & \leq  10^{50} 2^{2\vert e(X) \vert +2\vert e(Y) \vert} \vert e(X) \vert ( 10^{48} 2^{2\vert e(X) \vert +2\vert e(Y)\vert } \vert e(Y)\vert \deg_B(X))^6 \\
& = 10^{338}2^{14\vert e(X)\vert +14 \vert e(Y)\vert} \vert e(X)\vert \vert e(Y)\vert^6 \deg_B(X)^6.
\end{align*}
This concludes the proof of the theorem.
\end{proof}

\bibliography{refs}{}
\bibliographystyle{plain}

\end{document}